\documentclass[12pt]{amsart}

\usepackage{amssymb}
\usepackage{amsmath,amssymb,amsfonts,amsthm,latexsym}
\usepackage{mathrsfs}
\usepackage{color}
\usepackage{eucal}%    caligraphic-euler fonts: \mathcal{ }
\usepackage[all]{xypic}
\usepackage{xspace}
\usepackage[round]{natbib}
\usepackage{graphicx}
\usepackage{bbm}

\usepackage{graphics, amscd, amsfonts, epsfig, eepic, epic, psfrag, setspace}
\usepackage{xspace}

\newtheorem{theorem}{Theorem}
\newtheorem{corollary}{Corollary}

\newtheorem{lemma}{Lemma}

\numberwithin{equation}{section}

\usepackage{natbib}

\begin{document}
\doublespacing

\begin{center}
{\bf\Large Combinatorics of $\gamma$-structures}
\\
\vspace{15pt} Hillary S. W. Han$^1$, Thomas J. X. Li$^2$, Christian M. Reidys$^{\,\star}$
\end{center}

\begin{center}
         Department of Mathematics and Computer Science  \\
         University of Southern Denmark, Campusvej 55, \\
         DK-5230, Odense M, Denmark \\
         $^1$Phone: 45-20369907 \\
         $^1$email: hillaryswhan@gmail.com \\
         $^2$Phone: 45-91574526 \\
         $^2$email: thomasli@imada.sdu.dk\\
         Phone$^{\,\star}$: 45-24409251 \\
         email$^{\,\star}$: duck@santafe.edu \\
         Fax: 45-65502325 \\
         
\end{center}

\newpage
%:
\centerline{\bf Abstract}
In this paper we study canonical $\gamma$-structures, a class of RNA
pseudoknot structures that plays a key role in the context of polynomial
time folding of RNA pseudoknot structures.
A $\gamma$-structure is composed by specific building blocks, that have
topological genus less than or equal to $\gamma$, where composition means
concatenation and nesting of such blocks. Our main result is the derivation
of the generating function of $\gamma$-structures via symbolic enumeration
using so called irreducible shadows. We furthermore recursively compute the
generating polynomials of irreducible shadows of genus $\le \gamma$.
$\gamma$-structures are constructed via $\gamma$-matchings. For $1\le
\gamma \le 10$, we compute Puiseux-expansions at the unique, dominant
singularities, allowing us to derive simple asymptotic formulas for the
number of $\gamma$-structures.


\newpage
\section{Introduction and background}
\label{}
%%%
%%%%%%%%%%%%%%%%%%%%%%%%%%%%%%%%%%%%%%%%%%%%%%%%%%%%%%%%%%%%%%%%%%%%%%%%
%%%

An RNA sequence is a linear, oriented sequence of the nucleotides (bases)
{\bf A,U,G,C}. These sequences ``fold'' by establishing bonds between
pairs of nucleotides. These bonds cannot form arbitrarily, a nucleotide can
at most establish one bond and the global conformation of an RNA molecule
is determined by topological constraints encoded at the level of secondary
structure, i.e., by the mutual arrangements of the base pairs \cite{Bailor:10}.

Secondary structures can be interpreted as (partial) matchings in a graph of
permissible base pairs \cite{Tabaska:98}. When represented as a diagram,
i.e.~as a graph whose vertices are drawn on a horizontal line with arcs in
the upper halfplane on refers to a secondary structure with crossing arcs
as a pseudoknot structure.

Folded configurations exhibit the stacking of adjacent base pairs and
specific minimum arc-length conditions \cite{Waterman:78aa}, where
a stack is a sequence of parallel arcs $((i,j),(i+1,j-1),\dots,
(i+\tau,j-\tau))$.

The topological classification of RNA structures \cite{Bon:08,topmatch} has
recently been translated into an efficient folding algorithm
\cite{gfold}. This algorithm {\it a priori} folds into a novel class of
pseudoknot structures, the $\gamma$-structures.
$\gamma$-structures differ from pseudoknotted RNA structures of fixed
topological genus of an associated fatgraph or double line graph
\cite{Orland:02} and \cite{Bon:08}, since they have arbitrarily high
genus. They are composed by irreducible subdiagrams whose individual genus
is bounded by $\gamma$ and contain no bonds of length one ($1$-arcs),
see Section~\ref{S:facts} for details.

In \cite{NebelWeinberg} Nebel and Weinberg study a plethora of RNA structures.
The authors study asymptotic expansions for $\gamma=1$ and find that in the
limit of large $n$ there are $j_1\,n^{-\frac{3}{2}} (\varrho_{1,1}^{-1})^{n}$,
$1$-structures, where $\varrho_{1,1}^{-1}=3.8782$ and $j_1$ is some positive
constant.

In this paper we study canonical $\gamma$-structures, i.e.~partial matchings
composed by irreducible motifs of genus $\le \gamma$, without isolated arcs
and $1$-arcs. These motifs are called  irreducible shadows.
We first establish a functional relationship between the generating
function of $\gamma$-matchings and that of irreducible shadows.
Via this relation, we identify a polynomial $P_\gamma(u,X)$, whose
unique solution equals the generating function of $\gamma$-matchings.
We then derive a recurrence of the generating function of irreducible shadows
using Harer-Zagier recurrence \cite{Harer:86}.
The generating function of $\gamma$-matchings is then expanded at its unique
dominant singularity as a Puiseux-series. This implies, by means of
transfer theorems \cite{Flajolet:07a}, simple asymptotic formulas for the
numbers of $\gamma$-matchings.

$\gamma$-matchings are the stepping stone to derive via Lemma~\ref{L:GFbull}
the further refined, bivariate generating function of $\gamma$-shapes,
i.e.~$\gamma$-matchings containing only stacks composed by a single arc.
This generating function keeps additionally track of the $1$-arcs, that are
vital for the later inflation into $\gamma$-structures. We then
compute the generating function of $\tau$-canonical $\gamma$-structures
inflating $\gamma$-shapes by means of symbolic enumeration.

%%%
%%%%%%%%%%%%%%%%%%%%%%%%%%%%%%%%%%%%%%%%%%%%%%%%%%%%%%%%%%%%%%%%%%%%%%%%%%
%%%

%%%
%%%%%%%%%%%%%%%%%%%%%%%%%%%%%%%%%%%%%%%%%%%%%%%%%%%%%%%%%%%%%%%%%%%%%%%%%
%%%
\newpage
\section{Some basic facts}\label{S:facts}
%%%
%%%%%%%%%%%%%%%%%%%%%%%%%%%%%%%%%%%%%%%%%%%%%%%%%%%%%%%%%%%%%%%%%%%%%%%%%%
%%%

%%%
%%%%%%%%%%%%%%%%%%%%%%%%%%%%%%%%%%%%%%%%%%%%%%%%%%%%%%%%%%%%%%%%%%%%%%%%%%
%%%
\subsection{ $\gamma$-diagrams}
%%%
%%%%%%%%%%%%%%%%%%%%%%%%%%%%%%%%%%%%%%%%%%%%%%%%%%%%%%%%%%%%%%%%%%%%%%%%%%
%%%

A \emph{diagram} is a labeled graph over the vertex set $[n]=\{1, \dots, n\}$ in
which each vertex has degree $\le 3$, represented by drawing its vertices
in a horizontal line. The backbone of a diagram is the sequence of
consecutive integers $(1,\dots,n)$ together with the edges $\{\{i,i+1\}
\mid 1\le i\le n-1\}$. The arcs of a diagram, $(i,j)$, where $i<j$, are
drawn in the upper half-plane. We shall distinguish the backbone edge
$\{i,i+1\}$ from the arc $(i,i+1)$, which we refer to as a $1$-arc.

A \emph{stack} of length $\tau$ is a maximal sequence of
``parallel'' arcs,
$$
((i,j),(i+1,j-1),\dots,(i+\tau,j-\tau)).
$$
A stack of length $\ge \tau$ is called a $\tau$-canonical stack,
i.e.~a stack of length zero is an isolated arc.
The particular arc $(1,n)$ is called a rainbow and an arc is called
\emph{maximal} if it is maximal with respect to the partial order
$(i,j)\le (i',j')$ iff $i'\le i\;\wedge j\le j'$, see Fig.~\ref{F:stack}.

%%%%%%%%%%%%%%%%%%%%%%%%%%%%%%%%%%%%%%%%%%%%%%%%%%%%%%%%%%%%%%%%%%%%%%%
%%%%%%%%%%%%%%%%%%%%%%%%%%%%%%%%%%%%%%%%%%%%%%%%%%%%%%%%%%%%%%%%%%%%%%

\begin{figure}
  \includegraphics[width=0.95\textwidth]{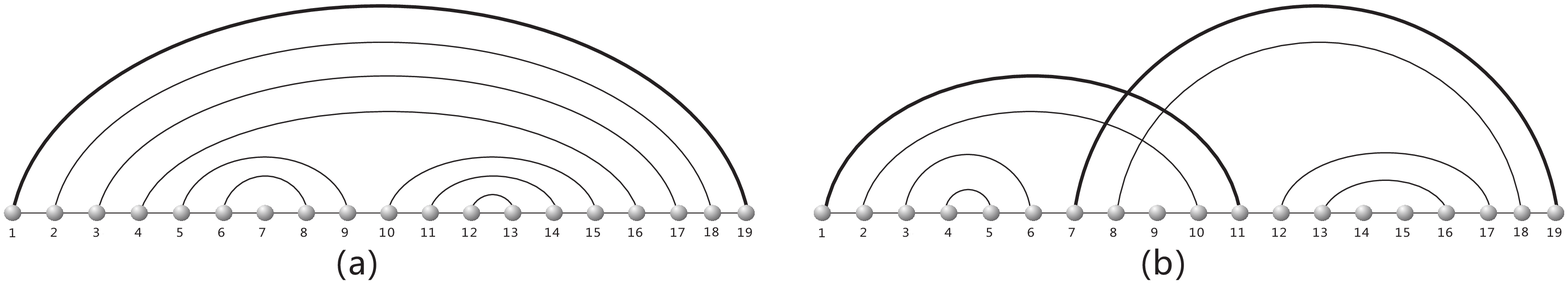}
\caption{\small $(a)$ a diagram containing a rainbow ({\bf bold}), three
stacks $((5,9), (6,8))$, $((10,15), (11,14), (12,13))$, and
$((1,19), (2, 18), (3,17),(4,16))$. $(b)$ the maximal arcs of a diagram
displayed in ({\bf bold}).}
\label{F:stack}      % Give a unique label
\end{figure}
%%%%%%%%%%%%%%%%%%%%%%%%%%%%%%%%%%%%%%%%%%%%%%%%%%%%%%%%%%%%%%%%%%

A stack of length $\tau$,
$((i,j),(i+1,j-1),\dots,(i+\tau,j-\tau))$ induces a sequence
of pairs $(([i, i+1],[j, j-1]),([i+1,i+2],[j-1,j-2])
\dots([i+\tau-1,i+\tau],[j-\tau,j-\tau+1]))$.
We call any of these $2\tau$ intervals a \emph{$P$-interval}.
The interval $[i+\tau,j-\tau]$ is called a \emph{$\sigma$-interval},
see Fig.~\ref{F:distinct}.

%%%%%%%%%%%%%%%%%%%%%%%%%%%%%%%%%%%%%%%%%%%%%%%%%%%%%%%%%%%%%%%%%%%%%%%%%%
%%%%%%%%%%%%%%%%%%%%%%%%%%%%%%%%%%%%%%%%%%%%%%%%%%%%%%%%%%%%%%%%%%%%%%

\begin{figure}
  \includegraphics[width=0.85\textwidth]{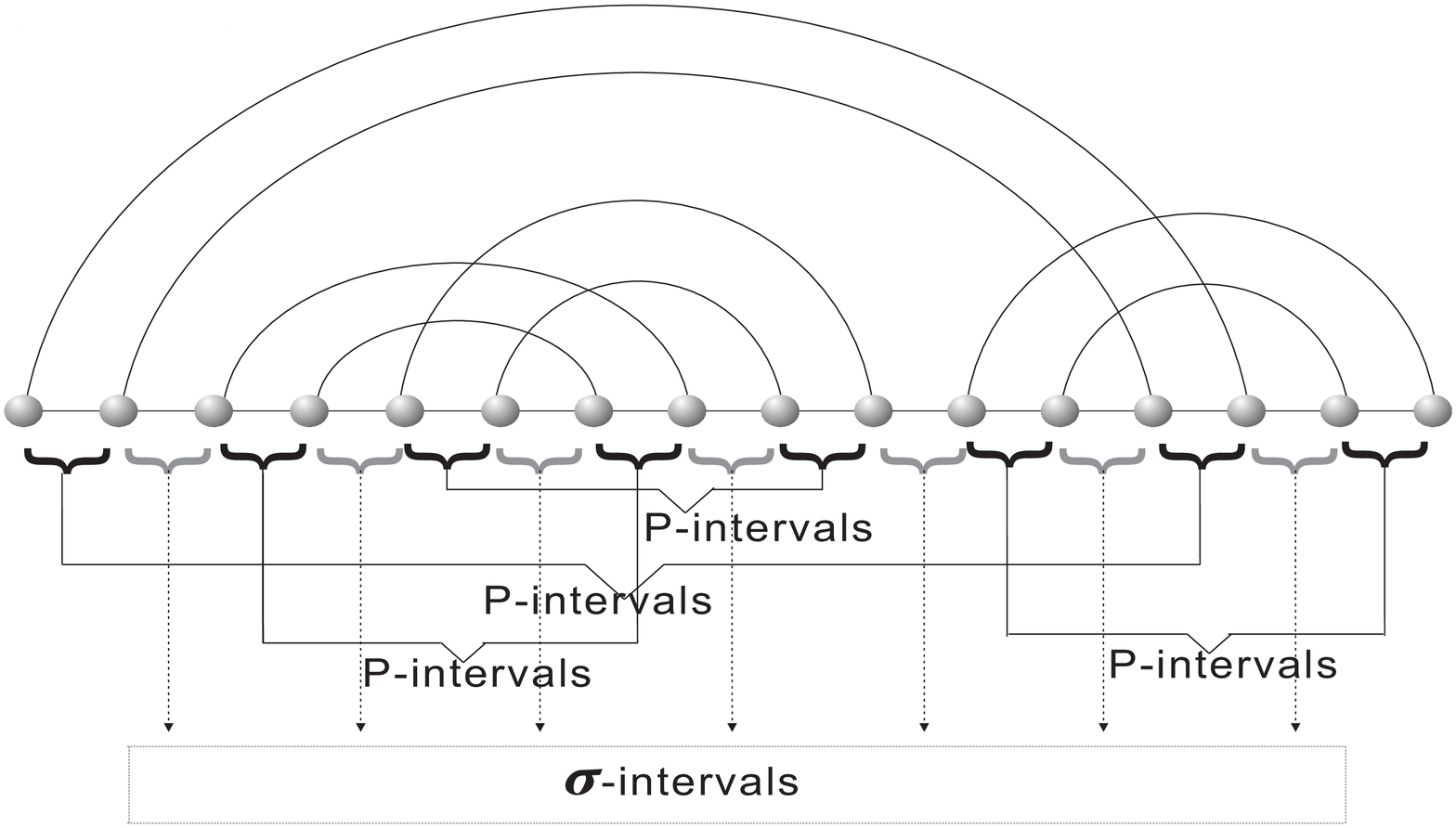}
\caption{\small $\sigma$- and $P$-intervals.}
\label{F:distinct}      % \
\end{figure}
%%%%%%%%%%%%%%%%%%%%%%%%%%%%%%%%%%%%%%%%%%%%%%%%%%%%%%%%%%%%%%%%%%
%%%%%%%%%%%%%%%%%%%%%%%%%%%%%%%%%%%%%%%%%%%%%%%%%%%%%%%%%%%%%%%%%%

We shall consider diagrams as fatgraphs, $\mathbb{G}$, that is graphs $G$
together with a collection of cyclic orderings, called fattenings, one such
ordering on the half-edges incident on each vertex.
Each fatgraph $\mathbb{G}$ determines an oriented surface $F(\mathbb{G})$
\cite{Loebl:08,Penner:10} which is connected if $G$ is and has some
associated genus $g(G)\ge 0$ and number $r(G)\ge 1$ of boundary components.
Clearly, $F(\mathbb{G})$ contains $G$ as a deformation retract \cite{Massey:69}.
Fatgraphs were first applied to RNA secondary structures in
\cite{Waterman:93} and \cite{Penner:03}.

A diagram $\mathbb{G}$ hence determines a unique surface $F(\mathbb{G})$
(with boundary). Filling the boundary components with discs we can
pass from $F(\mathbb{G})$ to a surface without boundary. Euler
characteristic, $\chi$, and genus, $g$, of this surface is given by
$\chi =  v - e + r$ and $g  =  1-\frac{1}{2}\chi$, respectively, where
$v,e,r$ is the number of discs, ribbons and boundary components in
$\mathbb{G}$, \cite{Massey:69}.
The \emph{genus} of a diagram is that of its associated surface without boundary.

The \emph{shadow} of a diagram of genus $g$ is obtained by removing all noncrossing
arcs, deleting all isolated vertices and collapsing all induced stacks
(i.e., maximal subsets of subsequent, parallel arcs) to single arcs, see
Fig.~\ref{F:shadow}. We denote shadows by $\sigma$.

%%%%%%%%%%%%%%%%%%%%%%%%%%%%%%%%%%%%%%%%%%%%%%%%%%%%%%%%%%%%%%%%%%%%%%%%%%
%%%
%%%%%%%%%%%%%%%%%%%%%%%%%%%%%%%%%%%%%%%%%%%%%%%%%%%%%%%%%%%%%%%%%%%%%%%%%%
\begin{figure}
  \includegraphics[width=0.95\textwidth]{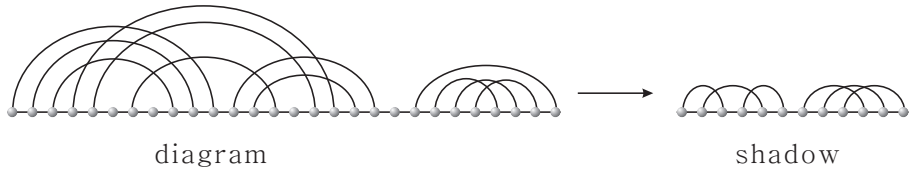}
\caption{\small Shadows: the shadow is obtained by removing all noncrossing
arcs and isolated points and collapsing all stacks and resulting
stacks into single arcs.}
\label{F:shadow}      % \
\end{figure}
%%%%%%%%%%%%%%%%%%%%%%%%%%%%%%%%%%%%%%%%%%%%%%%%%%%%%%%%%%%%%%%%%%
%%%%%%%%%%%%%%%%%%%%%%%%%%%%%%%%%%%%%%%%%%%%%%%%%%%%%%%%%%%%%%%%%%

The shadow of a diagram $\mathbb{G}$, $\sigma(\mathbb{G})$, can possibly
be empty. Furthermore, projecting into the shadow does not affect genus.
Any shadow of genus $g$ over one backbone contains at least $2g$ and at most
$(6g-2)$ arcs. In particular, for fixed genus $g$, there exist only finitely
many shadows \cite{gfold,fenix2bb}. In Fig.~\ref{F:4shadow}, we display
the four shadows of genus one.
%%%%%%%%%%%%%%%%%%%%%%%%%%%%%%%%%%%%%%%%%%%%%%%%%%%%%%%%%%%%%%%%%%%%%%%%%%
%%%%%%%%%%%%%%%%%%%%%%%
\begin{figure}
  \includegraphics[width=0.50\textwidth]{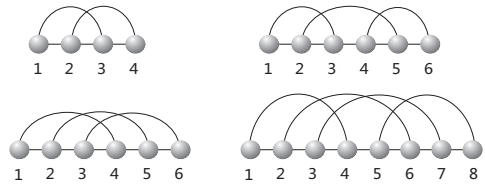}
\caption{\small The four shadows of genus one.}
\label{F:4shadow}      % \
\end{figure}

%%%%%%%%%%%%%%%%%%%%%%%%%%%%%%%%%%%%%%%%%%%%%%%%%%%%%%%%%%%%%%%%%%

A diagram is called \emph{irreducible}, if and only if for any two arcs,
$\alpha_1,\alpha_k$ contained in $E$, there exists a sequence of
arcs $(\alpha_1,\alpha_2,\dots,\alpha_{k-1},\alpha_k)$
such that $(\alpha_i,\alpha_{i+1})$ are crossing. Irreducibility
is equivalent to the concept of primitivity introduced by
\cite{Bon:08}, inspired by the work of \cite{Dyson:49b}.
According to
\cite{fenix2bb}, for arbitrary genus $g$ and $2g\le\ell\le (6g-2)$, there
exists an irreducible shadow of genus $g$ having exactly $\ell$ arcs.
We may reuse Fig.~\ref{F:4shadow} as an illustration of this result since
the four shadows of genus one are all irreducible.

Let ${\bf i}_g(m)$ denote the number of irreducible shadows of genus $g$
with $m$ arcs. Since for fixed genus $g$ there exist only finitely many
shadows we have the generating polynomial of irreducible shadows of genus
$g$
$$
{\bf I}_g(z)=\sum_{m=2g}^{6g-2}\,{\bf i}_g(m)z^m.
$$
For instance for genus $1$ and $2$ we have
\begin{eqnarray*}
{\bf I}_1(z) &=& {z}^{2} \left( 1+z \right) ^{2},\\
{\bf I}_2(z) &=& {z}^{4} \left( 1+z \right) ^{4} \left( 17+92\,z+96\,{z}^{2} \right).
\end{eqnarray*}

The shadow $\sigma(\mathbb{G})$ of a diagram $\mathbb{G}$ decomposes into
a set of irreducible shadows. We shall call these shadows \emph{irreducible
$\mathbb{G}$-shadows}.

Any diagram $\mathbb{G}$ can iteratively be decomposed by first removing
all noncrossing arcs as well as isolated vertices, second collapsing any
stacks and third by removing irreducible $\mathbb{G}$-shadows
iteratively as follows, see Fig.~\ref{F:decomposition}: \\
$\bullet$ one removes (i.e.~cuts the backbone at two points and after
          removal merges the cut-points) irreducible $\mathbb{G}$-shadows
          from bottom to top, i.e.~such that there exists no irreducible
          $\mathbb{G}$-shadow that is nested within the one previously
          removed. \\
$\bullet$ if the removal of an irreducible $\mathbb{G}$-shadow induces the
          formation of a stack, it is collapsed into a single arc.\\

%%%%%%%%%%%%%%%%%%%%%%%%%%%%%%%%%%%%%%%%%%%%%%%%%%%%%%%%%%%%%%%%%%%%%%%%%%
%%%
%%%
\begin{figure}
  \includegraphics[width=0.95\textwidth]{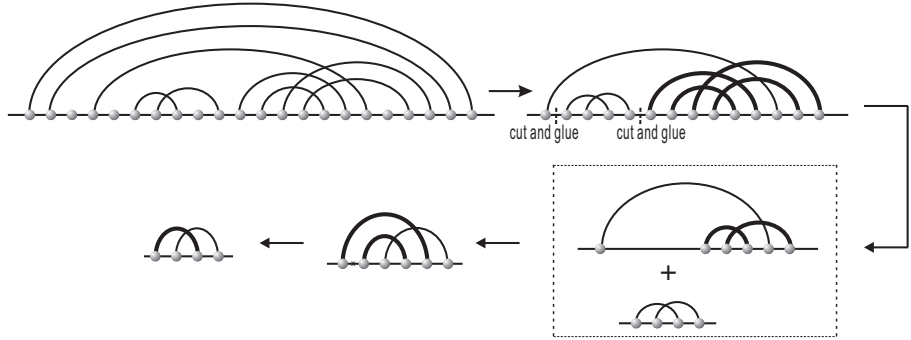}
\caption{\small A diagram $\mathbb{G}$ is decomposed:
we remove any noncrossing arcs and isolated points, collapse any
stacks into a single arcs and finally remove irreducible $\mathbb{G}$-shadows
from bottom to top and collapsing any stack generated in the process into a
single arc.}
\label{F:decomposition}     % \
\end{figure}

%%%%%%%%%%%%%%%%%%%%%%%%%%%%%%%%%%%%%%%%%%%%%%%%%%%%%%%%%%%%%%%%%%
%%%%%%%%%%%%%%%%%%%%%%%%%%%%%%%%%%%%%%%%%%%%%%%%%%%%%%%%%%%%%%%%%%

A diagram, $\mathbb{G}$, is a \emph{$\gamma$-diagram} if and only if for any
irreducible $\mathbb{G}$-shadow, $\mathbb{G}'$, $g(\mathbb{G}')\le
\gamma$ holds.

We denote the set of $\tau$-canonical $\gamma$-diagrams by
$\widetilde{\mathcal{G}}_{\tau,\gamma}$. Such a diagram without arcs of
the form $(i,i+1)$ ($1$-arcs) is called a \emph{$\tau$-canonical
$\gamma$-structure} and their set is denoted by $\mathcal{G}_{\tau,\gamma}$.
A \emph{$\gamma$-matching} is a $\gamma$-diagram that contains only vertices
of degree three.
A \emph{$\gamma$-shape} is a $\gamma$-matching that contains only stacks of
length zero.
Let $\mathcal{H}_{\gamma}$ and $\mathcal{S}_\gamma$ denote
the set of $\gamma$-matchings and $\gamma$-shapes, respectively.

%%%
%%%%%%%%%%%%%%%%%%%%%%%%%%%%%%%%%%%%%%%%%%%%%%%%%%%%%%%%%%%%%%%%%%%%%%%%%
%%%
\subsection{Some generating functions}
%%%%
%%%%%%%%%%%%%%%%%%%%%%%%%%%%%%%%%%%%%%%%%%%%%%%%%%%%%%%%%%%%%%%%%%%%%%%%
%%%%

In this paper we denote the ring of polynomials over a ring $R$ by
$R[X]$ and the ring of formal power series $\sum_{n\ge 0}a_nX^n$ by
$R[[X]]$. $R[[X]]$ is a local ring with maximal ideal $(X)$,
i.e.~any power series with nonzero constant term is invertible.
A Puiseux series \cite{C.T.C.Wall} is power series in fractional
powers of $X$, i.e.~$\sum_{n\ge 0}a_n{X^{n/k}}$ for some fixed
$k\in\mathbb{N}$.

We denote the generating functions of a set of diagrams $\mathcal{D}$
filtered by the number of arcs ${\bf D}(z)=\sum_{2n\ge 0}{\bf d}(n)z^n$.
Similarly, a generating function of diagrams filtered by the length of
the backbone is written as ${\bf D}(z)=\sum_{n\ge 0}{\bf d}(n)z^n$.
In particular, the generating functions of $\gamma$-matchings
and $\tau$ canonical $\gamma$-structures are given by
\begin{equation*}
{\bf H}_{\gamma}(u) = \sum_{2n\ge 0}{\bf h}_\gamma(n) u^n,\quad
{\bf G}_{\tau,\gamma}(z) = \sum_{n\ge 0}{\bf g}_{\tau,\gamma}(n) z^n.
\end{equation*}
Let $\mathcal H_{\gamma}(n,m)\supseteq \mathcal S_{\gamma}(n,m)$ denote the
collections of all $\gamma$-matchings and $\gamma$-shapes on $2n\geq
0$ vertices containing $m\geq 0$ $1$-arcs with generating functions
$$
{\bf H}_{\gamma}(x,y)=\sum_{m,2n\geq 0}{\bf h}_{\gamma}(n,m)x^ny^m,\quad
{\bf S}_{\gamma}(x,y)=\sum_{m,2n\geq 0}{\bf s}_{\gamma}(n,m)x^ny^m,
$$
where ${\bf h}_{\gamma}(n,m)={\bf s}_{\gamma}(n,m)=0$ if $2\gamma>n$
or if $m>n$.

Furthermore there is a natural projection $\vartheta$ from $\gamma$-matchings to
$\gamma$-shapes defined by collapsing each non-empty stack onto a single arc
$$
\vartheta\colon {\mathcal H}_\gamma\to  {\mathcal S}_\gamma,
$$
which is surjective and preserves irreducible shadows as well as the
number of $1$-arcs. $\vartheta$ restricts to a surjection
$$
\vartheta: \sqcup_{n\geq 0} {\mathcal H}_\gamma(n,m)\to\sqcup_{n\geq 0}
{\mathcal S}_\gamma(n,m),
$$
which collapses each stack to an arc and preserves any irreducible shadow
and also the number $m$ of $1$-arcs.

%%%
%%%%%%%%%%%%%%%%%%%%%%%%%%%%%%%%%%%%%%%%%%%%%%%%%%%%%%%%%%%%%%%%%%%%%%%%%%%
%%%

\newpage
\section{Combinatorics of $\gamma$-matchings}\label{S:combi}
%%%
%%%%%%%%%%%%%%%%%%%%%%%%%%%%%%%%%%%%%%%%%%%%%%%%%%%%%%%%%%%%%%%%%%%%%%%%%%
%%%

In this section we study $\gamma$-matchings.

%%%
%%%%%%%%%%%%%%%%%%%%%%%%%%%%%%%%%%%%%%%%%%%%%%%%%%%%%%%%%%%%%%%%%%%%%%%%
%%%
\begin{theorem}\label{T:H(u)}
Let $R=\mathbb{Z}[u]$. Then the following assertions hold:\\
{\bf (a)} the generating function of $\gamma$-matchings,
${\bf H}_{\gamma}(u)$, satisfies
\begin{equation}\label{E:canonical-stru}
{\bf H}_{\gamma}(u)^{-1}
 = 1- \left(u\,{\bf H}_{\gamma}(u)
 + {\bf H}_{\gamma}(u)^{-1}\sum_{g\le \gamma}\,
 {\bf I}_g\left(\frac{u\,{\bf H}_{\gamma}^{2}(u)}
 {1-u \,{\bf H}_{\gamma}^{2}(u)}\right) \right),
\end{equation}
equivalently,
\begin{equation*}
{\bf H}_{\gamma}(u)- u \,{\bf H}_{\gamma}(u)^2-
\sum_{g\le \gamma}\, {\bf I}_g\left(\frac{u\,{\bf H}_{\gamma}^{2}(u)} {1-u \,{\bf H}_{\gamma}^{2}(u)}\right)
=1.
\end{equation*}

In particular, there exists a polynomial $P_\gamma(u,X)\in R[X]$
of degree $(12\gamma -2)$, whose coefficients are sums
of ${\bf I}_g(z)$ coefficients, such that
$P_\gamma(u,{\bf H}_{\gamma}(u))=0$.\\
{\bf (b)} eq.~(\ref{E:canonical-stru}) determines ${\bf H}_{\gamma}(u)$
uniquely.
\end{theorem}
%%%
%%%%%%%%%%%%%%%%%%%%%%%%%%%%%%%%%%%%%%%%%%%%%%%%%%%%%%%%%%%%%%%%%%%%%%%%%
%%%
\begin{proof}
We first prove {\bf (a)}.
Let $\sigma$ be a fixed irreducible shadow of genus $g$ having $m$ arcs. Let
$\mathcal{V}_\sigma$ be the set of diagrams, generated by concatenating and
nesting $\sigma$.

{\it Claim 1:}
$$
{\bf V}_\sigma(u)= (1-{\bf V}_\sigma(u)^{-1}(u\,{\bf V}_\sigma(u)^{2})^m)^{-1}.
$$
To prove Claim $1$ we consider a $\mathcal{V}_\sigma$-diagram. Clearly, its
maximal arcs are contained in $t\ge 1$ copies of $\sigma$. These arcs induce
exactly $(2m-1)t$ $\sigma$-intervals, in each of which we find again an
element of $\mathcal{V}_\sigma$, whence
$$
{\bf V}_\sigma(u)=\sum_{t\ge 0} (u^m{\bf V}_\sigma(u)^{2m-1})^t
$$
and Claim $1$ follows.

Let $\mathcal{L}_\sigma$ be the set of diagrams having the fixed shape
$\sigma$ obtained by inflating $\sigma$-arcs into stacks, or symbolically,
$\mathcal{U}\,\times\textsc{Seq}(\mathcal{U})$. Here $\mathcal{U}$ and
$\mathcal{R}=\textsc{Seq}(\mathcal{U})$ denote the classes of arcs and
sequences of arcs. Clearly, the associated generating function of
$\mathcal{U}\times \mathcal{R}$ is
$u(1-u)^{-1}$.

Note that each $\mathcal{L}_\sigma$-diagram contains exactly
$(2m-1)$ $\sigma$-intervals and an arbitrary number of pairs of
${P}$-intervals.
Let $\mathcal{F}_\sigma$ denote the set of diagrams generated by
concatenating and nesting $\mathcal{L}_\sigma$-diagrams that contain
no empty $P$-intervals. Let finally $\mathcal{W}_\sigma$ be the set of
$1$-canonical diagrams, having shapes in $\mathcal{F}_\sigma$.

{\it Claim $2$.}
\begin{equation}\label{E:lstructure}
{\bf W}_\sigma(u)^{-1}=
1-{\bf W}_\sigma(u)^{-1}\left(\frac{\frac{u}{1-u}\,{\bf W}_\sigma^2(u)}
{1-\frac{u}{1-u}\,({\bf W}_\sigma^2(u)-1)}\right)^m.
\end{equation}
We shall construct $\mathcal{W}_\sigma$ using arcs, $\mathcal{U}$,
sequences of arcs, $\mathcal{R}$, induced arcs, $\mathcal{N}$,
and sequence of induced arcs, $\mathcal{M}$.
%%%%%%%%%%%%%%%%%%%%%%%%%%%%
%%%%%%%%%%%%%%%%%%%%%%%%%%%%%%%%%%%%%%%%%%%%%%%%%%%%%%%%%%%%%%%%%%%%%%%%%%
%%%
The class $\mathcal{F}_\sigma$ is obtained by concatenating and nesting
$\mathcal{L}_\sigma$-diagrams that do not contain any empty $P$-intervals,
see Fig.~\ref{F:V-K}.

%%%%%%%%%%%%%%%%%%%%%%%%%%%%%%%%%%%%%%%%%%%%%%%%%%%%%%%%%%%%%%%%%%%%%%%%%%
%%
%%%%%%%%%%%%%%%%%%%%%%%
\begin{figure}
  \includegraphics[width=0.95\textwidth]{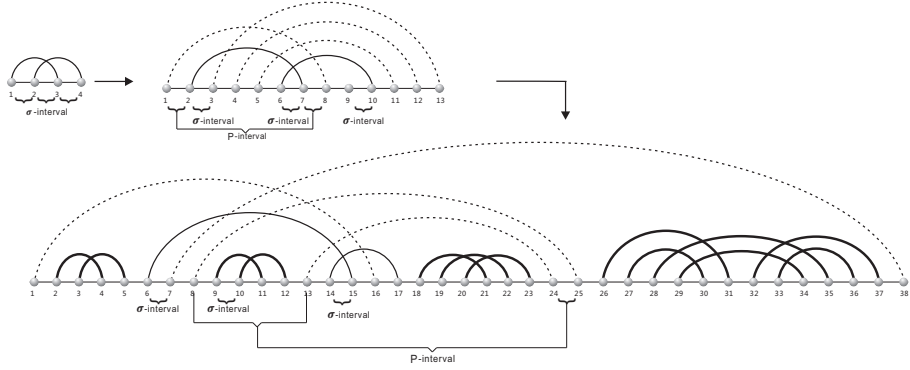}
\caption{\small First, a fixed irreducible shadow $\sigma$ is inflated into a
$\mathcal{L}_\sigma$-diagram, second we pass to an
$\mathcal{F}_\sigma$-diagram by inserting a nontrivial
$\mathcal{L}_\sigma$-diagram in one of the ${P}$-intervals.}
\label{F:V-K}      % \
\end{figure}

%%%%%%%%%%%%%%%%%%%%%%%%%%%%%%%%%%%%%%%%%%%%%%%%%%%%%%%%%%%%%%%%%%
An induced arc, i.e.~an arc together with at
least one nontrivial $\mathcal{F}_\sigma$-diagram in either one or
in both ${P}$-intervals
\begin{equation*}
\mathcal{N} = \mathcal{U}\times
\left( (\mathcal{F}_\sigma-1)+(\mathcal{F}_\sigma-1)+
(\mathcal{F}_\sigma-1)^2\right)
=\mathcal{U}\times \left(\mathcal{F}_\sigma^2-1\right).
\end{equation*}
Clearly, we have for a single induced arc
$\mathbf{N}(u)=u \left({\bf F}_\sigma(u)^2-1 \right)$
and for a sequence of induced arcs, $\mathcal{M}=
\textsc{Seq}(\mathcal{N})$, where
\begin{eqnarray*}
{\bf M}(u) & = &  \frac{1}{1-u \left({\bf F}_\sigma(u)^2-1 \right)}.
\end{eqnarray*}
By construction, the maximal arcs of an $\mathcal{F}_\sigma$-diagram
coincide with those of its underlying $\mathcal{V}_\sigma$-diagram.
Therefore
\begin{eqnarray*}
\mathcal{F}_\sigma&=&\sum_{t \geq 0}((\mathcal{U}\times \mathcal{M})^m \,
\mathcal{F}_\sigma^{2m-1})^t
\end{eqnarray*}
with generating function
\begin{equation}\label{E:gf-F}
{\bf F}_\sigma(u)=\sum_{t \geq 0}\left(\left( \frac{u}{1-u({\bf F}_\sigma(u)^2-1)}
\right)^m \,{\bf F}_\sigma(u)^{2m-1}\right)^t.
\end{equation}

Next we inflate the arcs of the $\mathcal{F}_\sigma$-diagram into
stacks, $\mathcal{U}\,\times \mathcal{R}$.

This inflation process generates $\mathcal{W}_\sigma$-diagrams and
any $\mathcal{W}_\sigma$-diagram can be constructed from a unique
fixed irreducible shadow $\sigma$ of genus $g$ with $m$ arcs.
We have
\begin{equation}\label{E:gf-W}
{\bf W}_\sigma(u)=\sum_{t \geq 0}\left(\left( \frac{\frac{u}{1-u}}
{1- \frac{u}{1-u}({\bf W}_\sigma(u)^2-1)}\right)^m\,
{\bf W}_\sigma(u)^{2m-1}\right)^t,
\end{equation}
whence Claim $2$.\\

%%%
%%%%%%%%%%%%%%%%%%%%%%%%%%%%%%%%%%%%%%%%%%%%%%%%%%%%%%%%%%%%%%%%%%%%%%%%%%%%%%%%%%%
%%%
{\it Claim 3:}
Let $M$ be the set of irreducible shadows of genus $g\le \gamma$.
Then\\
\begin{equation}\label{E:combi-fla}
{\bf W}_M(u)^{-1}= 1-\sum_{g \leq \gamma}\sum_{1<m}
{{\bf i}_g(m)\,\bf W}_M(u)^{-1}\left(\frac{u\,
{\bf W}_M^2(u)}{1-u\,{\bf W}_M^2(u)}\right)^m.
\end{equation}
%%%
%%%%%%%%%%%%%%%%%%%%%%%%%%%%%%%%%%%%%%%%%%%%%%%%%%%%%%%%%%%%%%%%%%%%%%%%%%%%%%%%%%%%%
%%%
The maximal arcs of a $\mathcal{V}_M$-structure, partition into
the maximal arcs of $t$ concatenated irreducible shadows
$\sigma_1,\dots,\sigma_t$ and
\begin{equation}\label{E:multinomial}
\sum_{\{\sigma_{1},\dots,\sigma_{t}\}
\atop \sigma_{i}\in M} 1 =
\left( \sum_{g \leq \gamma}\sum_{1<m}{\bf i}_g(m) \right)^t.
\end{equation}
These maximal arcs induce exactly $(2m-1)\,t$ $\sigma$-intervals.
In each $\sigma$-interval, we find again an element of $\mathcal{V}_M$.
Thus for any $\sigma_i$ having $m$ arcs, we have $\mathcal{V}_M^{2m-1}$,
which leads to the term $u^m{\bf V}_M(u)^{2m-1}$. It remains to
sum over all $t$, i.e.~expressing all the decompositions of
$\mathcal{V}_{M_\gamma}$-structures into concatenated, irreducible shadows
and we obtain
\begin{equation}\label{E:combi-irredu}
{\bf V}_M(u)=\sum_{t\ge 0} \left(\sum_{g \leq \gamma}
\sum_{m>1} {\bf i}_g(m) u^m{\bf V}_M(u)^{2m-1}\right)^t.
\end{equation}
The passage to from ${\mathcal V}_M$ to ${\mathcal L}_M$ as well as that
from ${\mathcal L}_M$ to ${\mathcal F}_M$ follows from Claim $2$, whence
\begin{equation}\label{E:F_M}
{\bf F}_M(u)= \sum_{t \geq 0}\left(\sum_{g \leq \gamma}\sum_{m>1}
{{\bf i}_g(m)\,\bf F}_M(u)^{-1}\left(\frac{u\,
{\bf F}_M^2(u)}{1-u \,({\bf F}_M^2(u)-1)}\right)^m \right)^t.
\end{equation}
Here ${\bf F}_M(u)^{-1}$ exists in $\mathbb{C}[[u]]$, having a nonzero
constant term.
Next we inflate the arcs of the $\mathcal{F}_M$-structure into
stacks, obtaining
\begin{equation}
{\bf W}_M(u)^{-1}= 1-\sum_{g \leq \gamma}\sum_{1<m}
{{\bf i}_g(m)\,\bf W}_M(u)^{-1}\left(\frac{u\,
{\bf W}_M^2(u)}{1-u\,({\bf W}_M^2(u))}\right)^m.
\end{equation}

We next derive the functional equation for $\mathbf{H}_{\gamma}(u)$
by incorporating noncrossing arcs. Since the maximal arcs composed
of noncrossing arcs are exactly rainbows, the generating function
of $\mathcal{H}_{\gamma}$-diagrams nested in a rainbow is given by
$u\, {\bf H}_{\gamma}(u)$. As in Claim $3$ we conclude
$$
{\bf H}_{\gamma}(u)^{-1}=
1-\sum_{g \leq \gamma}\left(u\,
{\bf H}_{\gamma}(u)+{\bf H}_{\gamma}(u)^{-1}\,\sum_{m>1}
{\bf i}_g(m)\,
\vartheta(u)^m \right),
$$
where
$$
\vartheta(u)=\frac{u\,{\bf H}_{\gamma}^{2}(u)}
 {1-u\,{\bf H}_{\gamma}^{2}(u)}.
$$
Setting $w_u(X)=1-u\,X^2$, eq.~(\ref{E:canonical-stru}) gives rise to
the polynomial
\begin{equation}\label{E:erni}
P_\gamma(u,X)= w_u(X)^{\kappa_\gamma} (-1 + X - u\,X^2)  -
\sum_{g \leq \gamma}\,w_u(X)^{\kappa_\gamma}
 \,{\bf I}_g\left(\frac{u\,X^{2}}
 {w_u(X)}\right),
\end{equation}
where $\kappa_\gamma=6\gamma-2$, $\text{\rm deg}(P_\gamma(u,X))
=  (2+2\kappa_\gamma)$,
$\left[X^{2+2\kappa_\gamma}\right]P_\gamma(u,X)  =   -u^{1+\kappa_\gamma}$ and
$P_\gamma(u,{\bf H}_\gamma(u)) =0$, whence {\bf (a)}.

It remains to prove {\bf (b)}. Since $M$ is the finite set of
irreducible shadows of genus $g\le\gamma$ and any such shadow
has $2g\le m\le \kappa_\gamma$ arcs \cite{fenix2bb}, any $M$-shadow has
$\le \kappa_\gamma$ arcs. Setting
$v(u)=1-u\,{\bf H}_{\gamma}^{2}(u)$, eq.~(\ref{E:canonical-stru})
implies
\begin{eqnarray*}
v(u)^{\kappa_\gamma} & = &
{\bf H}_{\gamma}(u)v(u)^{\kappa_\gamma}
- u\,{\bf H}^2_{\gamma}(u)\,v(u)^{\kappa_\gamma}  -
\sum_{g \leq \gamma}\,v(u)^{\kappa_\gamma}
 \,{\bf I}_g\left(\frac{u\,{\bf H}_{\gamma}^{2}(u)}
 {v(u)}\right)
\end{eqnarray*}
and consequently
\begin{equation}\label{E:canonical-stru5}
 \begin{split}
{\bf H}_{\gamma}(u)
&= -{\bf H}_{\gamma}(u)\sum_{i=1}^{\kappa_\gamma}\,
{\kappa_\gamma \choose i}1^{\kappa_\gamma-i}
{(v(u)-1)^i}+u\,{\bf H}^2_{\gamma}(u)\,v(u)^{\kappa_\gamma}+v(u)^{\kappa_\gamma}\\
&+ \sum_{g \leq \gamma}\,v(u)^{\kappa_\gamma}
{\bf I}_g\left(\frac{u\,{\bf H}_{\gamma}^{2}(u)}
  {v(u)}\right).
\end{split}
\end{equation}
All coefficients of ${\bf H}_{\gamma}(u)$ in the RHS of
eq.~(\ref{E:canonical-stru5}), are polynomials in $u$ of
degree $\geq 1$, whence any $[z^n]{\bf H}_{\gamma}(u)$ for
$n\ge (\kappa_\gamma+1)$ can be recursively computed.
Accordingly, eq.~(\ref{E:canonical-stru5}) determines
${\bf H}_{\gamma}(u)$ uniquely.
\end{proof}

%%%
%%%%%%%%%%%%%%%%%%%%%%%%%%%%%%%%%%%%%%%%%%%%%%%%%%%%%%%%%%%%%%%%%%%%%%%%%%%
%%%

\newpage
\section{Irreducible shadows}\label{S:irr}
%%%
%%%%%%%%%%%%%%%%%%%%%%%%%%%%%%%%%%%%%%%%%%%%%%%%%%%%%%%%%%%%%%%%%%%%%%%%%%
%%%
The bivariate generating function of irreducible shadows of genus $g$
with $m$ arcs is denoted by
\[
{\bf I}(z,t)=\sum_{g\geq 1}{\bf I}_g(z)\,t^g=\sum_{g\geq 1}\sum_{m=2g}^{6g-2}\,{\bf i}_g(m) \, z^m t^g .
\]
Let ${\bf c}_g(m)$ denote the number of matchings of genus $g$
with $m$ arcs. We have the generating function of matchings of genus
$g$
$$
{\bf C}_g(z)=\sum_{m\geq 2g}\,{\bf c}_g(m) z^m.
$$
The bivariate generating function of matchings of genus $g$
with $m$ arcs is denoted by
\[
{\bf C}(z,t)=\sum_{g\geq 0}{\bf C}_g(z)\,t^g=\sum_{g\geq 0}\sum_{m\geq 2g}\,{\bf c}_g(m)\, z^m t^g .
\]

%%%
%%%%%%%%%%%%%%%%%%%%%%%%%%%%%%%%%%%%%%%%%%%%%%%%%%%%%%%%%%%%%%%%%%%%%%%%
%%%
\begin{theorem}\label{T:CIrelation}
The generating functions  ${\bf C}(z,t)$ and ${\bf I}(z,t)$ satisfy
\begin{equation*}
{\bf C}(z,t)^{-1} = 1- \left(z\,{\bf C}(z,t)
 + {\bf C}(z,t)^{-1} {\bf I}\left(\frac{z\,{\bf C}(z,t)^2}{1-z\,{\bf C}(z,t)^2},t\right)\right),
\end{equation*}
equivalently,
\begin{equation}\label{E:CIrelation}
{\bf C}(z,t)-z\,{\bf C}(z,t)^2- {\bf I}\left(\frac{z\,{\bf C}(z,t)^2}{1-z\,{\bf C}(z,t)^2},t\right)=1.
\end{equation}
\end{theorem}
%%%
%%%%%%%%%%%%%%%%%%%%%%%%%%%%%%%%%%%%%%%%%%%%%%%%%%%%%%%%%%%%%%%%%%%%%%%%%
%%%
\begin{proof}
We distinguish the classes of blocks into two categories characterized by
the unique component containing all maximal arcs (maximal component). Namely,\\
$\bullet$ blocks whose maximal component contains only one arc, \\
$\bullet$ blocks whose maximal component is an (nonempty) irreducible matching.\\
In the first case, the removal of the maximal component (one arc) generates
again an arbitrary matching, which translates into the term
$$
z\,{\bf C}(z,t).
$$
Let ${\bf T}(z,t)$ denote the (genus filtered) generating function of blocks of
the second type. The decomposition of matchings into a sequence of
blocks implies
\[
{\bf C}(z,t)^{-1}
 = 1- \left( z\,{\bf C}(z,t)  + {\bf T}(z,t)\right).
\]
Let $\sigma$ be a fixed irreducible shadow of genus $g$ having $n$ arcs. Let
${\bf T}_\sigma(z,t)$ be the generating function of blocks, having $\sigma$ as
the shadow of its unique maximal component. Then we have
\[
{\bf T}(z,t)= \sum_{\sigma \in \mathcal{I}}{\bf T}_\sigma(z,t),
\]
where $\mathcal{I}$ denotes the set of irreducible shadows.

We shall construct $\mathcal{T}_\sigma$ in three steps using arcs, $\mathcal{R}$,
sequences of arcs, $\mathcal{K}$, induced arcs, $\mathcal{N}$, sequence of
induced arcs, $\mathcal{M}$, and arbitrary matchings, $\mathcal{C}$.

{\bf Step I:} We inflate each arc in $\sigma$ into a sequence of induced arcs,
see Fig.~\ref{F:H1}. An induced arc, i.e.~an arc together with at least one
nontrivial matching in either one or in both ${P}$-intervals
\begin{equation*}
\mathcal{N} = \mathcal{R}\times
\left( (\mathcal{C}-1)+(\mathcal{C}-1)+
(\mathcal{C}-1)^2\right)
=\mathcal{R}\times \left(\mathcal{C}^2-1\right).
\end{equation*}
%%%
%%%%%%%%%%%%%%%%%%%%%%%%%%%%%%%%%%%%%%%%%%%%%%%%%%%%%%%%%%%%%%%%%%%%%%%%%
%%%
\begin{figure}
\begin{center}
\includegraphics[width=0.8\textwidth]{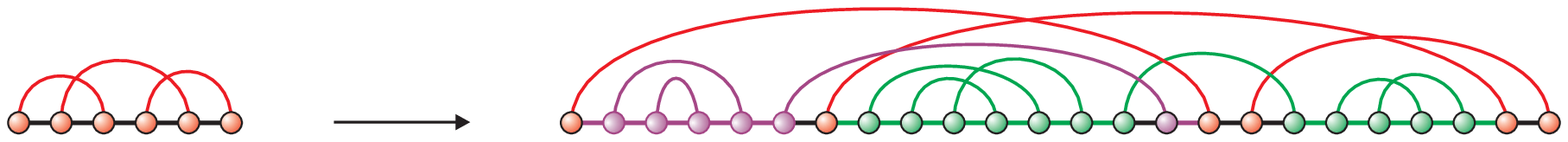}
\end{center}
\caption{\small {\bf Step I:} inflation of each arc in $\sigma$ into a sequence of induced arcs.
}\label{F:H1}
\end{figure}
%%%
%%%%%%%%%%%%%%%%%%%%%%%%%%%%%%%%%%%%%%%%%%%%%%%%%%%%%%%%%%%%%%%%%%%%%%%%%
%%%
Clearly, we have for a single induced arc
$\mathbf{N}(z,t)=z \left({\bf C}(z,t)^2-1 \right)$, guaranteed
by the additivity of genus,
and for a sequence of induced arcs, $\mathcal{M}=
\textsc{Seq}(\mathcal{N})$, where
\begin{eqnarray*}
{\bf M}(z,t) & = &  \frac{1}{1-z \left({\bf C}(z,t)^2-1 \right)}.
\end{eqnarray*}
Inflating each arc into a sequence of induced arcs, $R^n \times \mathcal{M}^n$,
gives the corresponding generating function
\[
z^n {\bf M}(z,t)^n=\left(\frac{z}{1-z \left({\bf C}(z,t)^2-1 \right)}\right)^n,
\]
since the genus is additive.

{\bf Step II:} We inflate each arc in the component with shadow $\sigma$ into
stacks, see Fig.~\ref{F:H2}. The corresponding generating function is
\[
\left(\frac{\frac{z}{1-z}}{1-\frac{z}{1-z}
\left({\bf C}(z,t)^2-1 \right)}\right)^n=
\left(\frac{z}{1-z {\bf C}(z,t)^2}\right)^n
\]
%%%
%%%%%%%%%%%%%%%%%%%%%%%%%%%%%%%%%%%%%%%%%%%%%%%%%%%%%%%%%%%%%%%%%%%%%%%%%
%%%
\begin{figure}
\begin{center}
\includegraphics[width=0.9\textwidth]{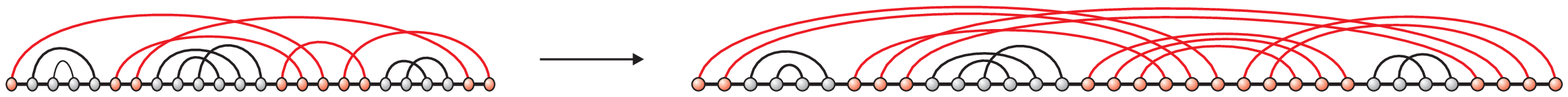}
\end{center}
\caption{\small  {\bf Step II:} inflation of each arc in the component with
shadow $\sigma$ into stacks.
}\label{F:H2}
\end{figure}
%%%
%%%%%%%%%%%%%%%%%%%%%%%%%%%%%%%%%%%%%%%%%%%%%%%%%%%%%%%%%%%%%%%%%%%%%%%%%
%%%
{\bf Step III:} We insert additional matchings at exactly
$(2n-1)$ $\sigma$-intervals, see Fig.~\ref{F:H3}. Accordingly, the
generating function is ${\bf C}(z,t)^{2n-1}$.
%%%
%%%%%%%%%%%%%%%%%%%%%%%%%%%%%%%%%%%%%%%%%%%%%%%%%%%%%%%%%%%%%%%%%%%%%%%%%
%%%
\begin{figure}
\begin{center}
\includegraphics[width=1\textwidth]{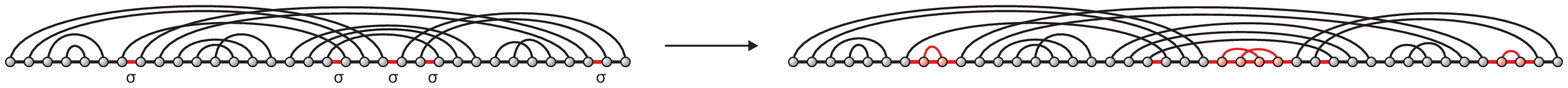}
\end{center}
\caption{\small  {\bf Step III:} insertion of additional matchings at exactly
$(2n-1)$ $\sigma$-intervals.
}\label{F:H3}
\end{figure}
%%%
%%%%%%%%%%%%%%%%%%%%%%%%%%%%%%%%%%%%%%%%%%%%%%%%%%%%%%%%%%%%%%%%%%%%%%%%%
%%%

Combining these three steps and utilizing additivity of the genus, we arrive at
\begin{eqnarray*}
{\bf T}_\sigma(z,t)&=& t^g \left(\frac{z}{1-z {\bf C}(z,t)^2}\right)^n
{\bf C}(z,t)^{2n-1}\\
&=& t^g {\bf C}(z,t)^{-1}\,
      \left(\frac{z {\bf C}(z,t)^2}{1-z {\bf C}(z,t)^2}\right)^n.
\end{eqnarray*}
Therefore
\begin{eqnarray*}
{\bf T}(z,t)&=& \sum_{\sigma \in \mathcal{I}}{\bf T}_\sigma(z,t)\\
&=& \sum_{g,n} {\bf i}_g(n) t^g \,{\bf C}(z,t)^{-1}\,
\left(\frac{z {\bf C}(z,t)^2}{1-z {\bf C}(z,t)^2}\right)^n.
\end{eqnarray*}
We derive
\[
{\bf T}(z,t)={\bf C}(z,t)^{-1} {\bf I}
\left(\frac{z\,{\bf C}(z,t)^2}{1-z\,{\bf C}(z,t)^2},t\right),
\]
completing the proof of eq.~(\ref{E:CIrelation}).

\end{proof}

Now we can derive a recursion for ${\bf I}_g(z)$ from Theorem~\ref{T:CIrelation}.
%%%
%%%%%%%%%%%%%%%%%%%%%%%%%%%%%%%%%%%%%%%%%%%%%%%%%%%%%%%%%%%%%%%%%%%%%%%%
%%%
\begin{corollary}\label{C:irr}
For $g\ge 1$, ${\bf I}_g(z)$ satisfies the following recursion
\begin{equation*}
\begin{split}
& {\bf I}_g(z)={\bf C}_g(\theta(z))-\theta(z)\, \sum_{i=0}^g{\bf C}_i(\theta(z)){\bf C}_{g-i}(\theta(z)) \\
& - \sum_{j=1}^{g-1} [t^{g-j}]{\bf I}_j\left(\frac{\theta(z)\,(\sum_{k=0}^{g-j}{\bf C}_k(\theta(z)) t^k)^2}{1-\theta(z)\,(\sum_{k=0}^{g-j}{\bf C}_k(\theta(z)) t^k)^2}\right),
\end{split}
\end{equation*}
where $\theta(z)=\frac{z (z+1)}{(2z+1)^2}$.
\end{corollary}
%%%
%%%%%%%%%%%%%%%%%%%%%%%%%%%%%%%%%%%%%%%%%%%%%%%%%%%%%%%%%%%%%%%%%%%%%%%%%
%%%
\begin{proof}
We compute the coefficient of $t^g$ on both sides of eq.~(\ref{E:CIrelation})
\begin{equation*}
[t^g]{\bf C}(t,z)-z\,[t^g] {\bf C}(t,z)^2- [t^g]{\bf I}\left(\frac{z\,{\bf C}(t,z)^2}{1-z\,{\bf C}(t,z)^2},t\right)=0
\end{equation*}
\begin{equation*}
{\bf C}_g(z)-z\, \sum_{i=0}^g{\bf C}_i(z){\bf C}_{g-i}(z)- \sum_{j=1}^{g} [t^{g-j}]{\bf I}_j\left(\frac{z\,{\bf C}(t,z)^2}{1-z\,{\bf C}(t,z)^2}\right)=0
\end{equation*}
\begin{equation*}
\begin{split}
& {\bf C}_g(z)-z\, \sum_{i=0}^g{\bf C}_i(z){\bf C}_{g-i}(z)- \sum_{j=1}^{g-1} [t^{g-j}]{\bf I}_j\left(\frac{z\,{\bf C}(t,z)^2}{1-z\,{\bf C}(t,z)^2}\right) \\
& =[t^{0}]{\bf I}_g\left(\frac{z\,{\bf C}(t,z)^2}{1-z\,{\bf C}(t,z)^2}\right)
\end{split}
\end{equation*}
Note that
\begin{equation*}
[t^{g-j}]{\bf I}_j\left(\frac{z\,{\bf C}(t,z)^2}{1-z\,{\bf C}(t,z)^2}\right)=[t^{g-j}]{\bf I}_j\left(\frac{z\,(\sum_{k=0}^{g-j}{\bf C}_k(z) t^k)^2}{1-z\,(\sum_{k=0}^{g-j}{\bf C}_k(z) t^k)^2}\right).
\end{equation*}
Hence,
\begin{equation*}
\begin{split}
& {\bf C}_g(z)-z\, \sum_{i=0}^g{\bf C}_i(z){\bf C}_{g-i}(z)- \sum_{j=1}^{g-1} [t^{g-j}]{\bf I}_j\left(\frac{z\,(\sum_{k=0}^{g-j}{\bf C}_k(z) t^k)^2}{1-z\,(\sum_{k=0}^{g-j}{\bf C}_k(z) t^k)^2}\right) \\
& ={\bf I}_g\left(\frac{z\,{\bf C}_0(z)^2}{1-z\,{\bf C}_0(z)^2}\right) 
\end{split}
\end{equation*}
Setting $y=\frac{z\,{\bf C}_0(z)^2}{1-z\,{\bf C}_0(z)^2}$, we have $z=\theta(y)=\frac{y (y+1)}{(2y+1)^2}$. Then we derive
\begin{equation*}
\begin{split}
&{\bf I}_g(y)={\bf C}_g(\theta(y))-\theta(y)\, \sum_{i=0}^g{\bf C}_i(\theta(y)){\bf C}_{g-i}(\theta(y)) \\
&- \sum_{j=1}^{g-1} [t^{g-j}]{\bf I}_j\left(\frac{\theta(y)\,(\sum_{k=0}^{g-j}{\bf C}_k(\theta(y)) t^k)^2}{1-\theta(y)\,(\sum_{k=0}^{g-j}{\bf C}_k(\theta(y)) t^k)^2}\right)
\end{split}
\end{equation*}
completing the proof.
\end{proof}

A seminal result due to \cite{Harer:86}, computes a recursion and generating
function for the number ${\bf c}_g(m)$ as follows :

\begin{lemma}\cite{Harer:86}\label{L:recursion}
The ${\bf c}_g(m)$ satisfy the recursion
\begin{equation}\label{E:recursion}
(m+1)\, \mathbf{c}_g(m)  =  2(2m-1)\,\mathbf{c}_g(m-1)+
                          (2m-1)(m-1)(2m-3)\,\mathbf{c}_{g-1}(m-2),
\end{equation}
where $\mathbf{c}_g(m)=0$ for $2g>m$.
\end{lemma}

The recursion eq.~(\ref{E:recursion}) is equivalent to the ODE
\begin{eqnarray}\label{E:ODE}
z(1-4z)\frac{d}{dz}\mathbf{C}_g(z) +(1-2z)\mathbf{C}_g(z) & = &
\Phi_{g-1}(z),
\end{eqnarray}
where
\begin{eqnarray*}
\begin{split}
& \Phi_{g-1}(z) = \\
& z^2\left(4z^3\frac{d^3}{dz^3}\mathbf{C}_{g-1}(z) +
24z^2 \frac{d^2}{dz^2}\mathbf{C}_{g-1}(z) + 27z\frac{d}{dz}
\mathbf{C}_{g-1}(z)+3\mathbf{C}_{g-1}(z)\right)
\end{split}
\end{eqnarray*}
with initial condition $\mathbf{C}_g(0)=0$
since $r=n+1-2g$ has no positive solution $r>0$ for $n<2g$.
Therefore we can recursively compute $\mathbf{C}_g(z)$ by solving eq.~(\ref{E:ODE}) via \textsf{Maple}.
%%%
%%%%%%%%%%%%%%%%%%%%%%%%%%%%%%%%%%%%%%%%%%%%%%%%%%%%%%%%%%%%%%%%%%%%%%%%%
%%%
\begin{theorem}\label{E:GF} \cite{topmatch}
For any $g\ge 1$ the generating function ${\bf C}_g(z)$ is given by
\begin{eqnarray}\label{E:it}
\mathbf{C}_g(z) = \, Q_g(z)\frac{\sqrt {1-4\,z}}{(1-4z)^{3g}},
\end{eqnarray}
where $Q_g(z)$ is a polynomial with integral coefficients of degree at
most $(3g-1)$, $Q_g(1/4)\neq 0$, $[z^{2g}]Q_g(z)\neq 0$ and
$[z^h]Q_g(z)=0$ for $0\leq h\leq 2g-1$.
\end{theorem}
%%%
%%%%%%%%%%%%%%%%%%%%%%%%%%%%%%%%%%%%%%%%%%%%%%%%%%%%%%%%%%%%%%%%%%%%%%%%%
%%%
The recursion eq.~(\ref{E:ODE}) permits the calculation of the polynomials  $Q_g(z)$,
the first five of which are given as follows \cite{topmatch}
\begin{eqnarray*}
Q_1(z) &=& z^2,\\
Q_2(z) &=& 21z^4\, \left( z+1 \right)\\
Q_3(z) &=&  11z^6\, \left( 158\,{z}^{2}+558\,z+135 \right),\\
Q_4(z) &=&143z^8\left( 2339\,{z}^{3}+18378\,{z}^{2}+13689\,z+1575 \right),\\
Q_5(z) &=&  88179z^{10}\, \left( 1354\,{z}^{4}+18908\,
{z}^{3}+28764\,{z}^{2}+9660\,z+675 \right).
\end{eqnarray*}

Applying Corollary~\ref{C:irr} together with the generating function ${\bf C}_g(z)$,
we recursively compute ${\bf I}_g(z)$.

For example, for $g=1$,
\begin{eqnarray*}
{\bf I}_1(z) &=&{\bf C}_1(\theta(z))-2\theta(z)\,{\bf C}_0(\theta(z)){\bf C}_{1}(\theta(z))\\
             &=& {z}^{2}\left( 1+z \right) ^{2}.
\end{eqnarray*}

For $1\leq g \leq 8$, we list ${\bf I}_g(z)$ as follows
\begin{eqnarray*}
{\bf I}_1(z)&=&{z}^{2}\left( 1+z \right) ^{2},\\
{\bf I}_2(z)&=&{z}^{4} \left( 1+z \right) ^{4} \left( 17+92\,z+96\,{z}^{2} \right)\\
{\bf I}_3(z)&=&{z}^{6} \left( 1+z \right) ^{6} \left( 1259+15928\,z+61850\,{z}^{2}+
92736\,{z}^{3}+47040\,{z}^{4} \right)\\
{\bf I}_4(z)&=&{z}^{8} \left( 1+z \right) ^{8} \left( 200589+4245684\,z+31264164\,{z}
^{2}+107622740\,{z}^{3}\right.\\
&&\left.+188262816\,{z}^{4}+161967360\,{z}^{5}+54333440\,{z}^{6} \right)\\
{\bf I}_5(z) &=& {z}^{10} \left( 1+z \right) ^{10} \left( 54766516+1681752448\,z+
19092044658\,{z}^{2} \right.\\
&& +109184482584\,{z}^{3}+353376676011\,{z}^{4}\\
&&+ 675135053568\,{z}^{5}+753610999040\,{z}^{6}\\
&&\left.+453941596160\,{z}^{7}+
113867919360\,{z}^{8} \right)\\
{\bf I}_6(z)&=&3\,{z}^{12} \left( 1+z \right) ^{12} \left(7613067765+ 312905543772\,z+4932317894440\,{z}^{2}\right.\\
&&+40797413383380\,{z}^{3}+200964285178270\,{z}^{4}+626595744773516\,{z}^{5}\\
&&+1268150755326432\,{z}^{6}+
1660845652501760\,{z}^{7}+1357241056522240\,{z}^{8}\\
&&\left.+628740761518080\,{z}^{9}+126004558299136
\,{z}^{10} \right)
\end{eqnarray*}
\begin{eqnarray*}
{\bf I}_7(z)&=&{z}^{14} \left( 1+z \right) ^{14} \left( 13532959408258+
706557271551408\,z \right.\\
&&+14506513039164060\,{z}^{2}+160434554727348896\,{z}^{3}\\
&&+1089075339931680039\,{z}^{4} +4857650169218369856\,{z}^{5}\\
&&+14771712773087154704\,{z}^{6}+31138771188689736192\,{z}^{7}\\
&& +45486763075779571200\,{z}^{8}+45167296685229793280\,{z}^{9}\\
&& +29078583024627105792\,{z}^{10}\\
&&\left.+10941912454886326272\,{z}^{11}+1826131581135486976\,{z}^{12} \right)\\
{\bf I}_8(z)&=&{z}^{16} \left( 1+z \right) ^{16} \left( 10826939105517381+
692156096364848676\,z  \right.\\
&&+17724869034206737356\,{z}^{2} \\
&&+249069951630509297956\,{z}^{3}+2192230050291936695620\,{z}^{4}\\
&&+12980362620620450943588\,{z}^{5}+53923920139564145104556\,{z}^{6}\\
&&+161060520394034807160164\,{z}^{7}+349969438514715552162336\,{z}^{8}\\
&&+553647075623879302120960\,{z}^{9}+630641488385967162351616\,{z}^{10}\\
&&+503519879227179011162112\,{z}^{11}+267275771110990512783360\,{z}^{12}\\
&&\left.+84670509266097640833024\,{z}^{13}+12107536630199227514880\,{z}^{14}\right)\\
\end{eqnarray*}

We conjecture that the polynomial ${\bf I}_g(z)$, for arbitrary $g$, has ${z}^{2g}\left( 1+z \right) ^{2g}$ as a factor.
%%%%
%%%%%%%%%%%%%%%%%%%%%%%%%%%%%%%%%%%%%%%%%%%%%%%%%%%%%%%%%%%%%%%%%%%%%%%%
%%%%
\newpage
\section{Asymptotics of $\gamma$-matchings}

%%%%
%%%%%%%%%%%%%%%%%%%%%%%%%%%%%%%%%%%%%%%%%%%%%%%%%%%%%%%%%%%%%%%%%%
%%%%

Let us begin recalling the following result of \cite{Flajolet:07a}:

%%%%%%%%%%%%%%%%%%%%%%%%%%%%%%%%%%%%%%%%%%%%%%%%%%%%%%%%%%%%%%%%%%%%%
%%%
\begin{theorem}\label{T:AsymG}
Let $y(u)=\sum_{n\geq 0}y_n u^n$ be a generating function, analytic at $0$, satisfy
a polynomial equation $\Phi(u,y)=0$. Let $\rho$ be the real dominant singularity
of $y(u)$. Define the resultant of $\Phi(u,y)$ and $\frac{\partial}{\partial y}
\Phi(u,y)$ as polynomial in $y$
\[
\Delta(u)= \mathbf{R} \left(\Phi(u,y), \frac{\partial}{\partial y}\Phi(u,y),y
\right).
\]
{\bf(1)} The dominant singularity $\rho$ is unique and a root of the
resultant $\Delta(u)$ and there exists $\pi=y(\rho)$, satisfying the system of
equations,
\begin{equation}\label{E:phi1}
\Phi(\rho,\pi)=0,\quad \Phi_y(\rho,\pi)=0.
\end{equation}
{\bf(2)} If $\Phi(u,y)$ satisfies the conditions:
\begin{equation}\label{E:phi2}
\Phi_u(\rho,\pi)\neq 0,\quad \Phi_{y y}(\rho,\pi)\neq 0,
\end{equation}
then $y(u)$ has the following expansion at $\rho$
\begin{equation}\label{E:asym1}
y(u)=\pi+ \lambda (\rho-u)^{\frac{1}{2}}+O(\rho-u),\quad \text{for some nonuero
constant } \lambda.
\end{equation}
Further the coefficients of $y(u)$ satisfy
\[
[u^n]y(u) \sim  c \, n^{-\frac{3}{2}} \rho^{-n}, \quad n\rightarrow \infty,
\]
for some constant $c>0$.
\end{theorem}
%%%
%%%%%%%%%%%%%%%%%%%%%%%%%%%%%%%%%%%%%%%%%%%%%%%%%%%%%%%%%%%%%%%%%%%%%
%%%
\begin{proof}
The proof of {\bf(1)} can be found in \cite{Flajolet:07a} or \cite{Hille:62}
pp. 103.
To prove {\bf(2)}, let $\Psi(u,y)=\Phi(\rho-u,\pi-y)$. Immediately, we have
$\Psi(0,0)=0$. Puiseux's Theorem~\cite{C.T.C.Wall} guarantees a solution of
$y-\pi$ in terms of a Puiseux series in $\rho-u$.
Note that equations~(\ref{E:phi1}) and (\ref{E:phi2}) are equivalent to
\[
\Psi(0,0)=0,\quad \Psi_y(0,0)=0,\quad \Psi_u(0,0)\neq 0,\quad \Psi_{y y}(0,0)\neq 0.
\]
Then we apply Newton's polygon method to determine the type of expansion and
find the first exponent of $u$ to be $\frac{1}{2}$.
Therefore the Puiseux series expansion of $y(u)$ has the
required form. The asymptotics of the coefficients follows
from eq.~(\ref{E:asym1}) as a straightforward application of the transfer
theorem (\cite{Flajolet:07a}, pp. 389 Theorem VI.3).
\end{proof}

Combining Theorem~\ref{T:H(u)} and Theorem~\ref{T:AsymG}, the asymptotic
analysis of ${\bf H}_{\gamma}(u)$ follows.
%%%
%%%%%%%%%%%%%%%%%%%%%%%%%%%%%%%%%%%%%%%%%%%%%%%%%%%%%%%%%%%%%%%%%%%%%
%%%
\begin{theorem}\label{T:HUniAsy}
For $1\leq \gamma \leq 10$, let
\[
\Delta_\gamma(u)= \mathbf{R} \left(P_\gamma(u,X), \frac{\partial}{\partial X}P_\gamma(u,X),X \right)
\]
the resultant of $P_\gamma(u,X)$ and $\frac{\partial}{\partial X}P_\gamma(u,X)$ as polynomials in $X$,
and $\rho_{\gamma}$ denote the real dominant singularity of ${\bf H}_{\gamma}(u)$. \\
{\bf (a)} the dominant singularity $\rho_{\gamma}$ is unique and a root of $\Delta_\gamma(u)$, \\
{\bf (b)} at $\rho_{\gamma}$ we have
\begin{equation*}
{\bf H}_{\gamma}(u) = \pi_\gamma  +  \lambda_\gamma (\rho_{\gamma}-u)^{\frac{1}{2}}+O(\rho_{\gamma}-u),\quad \text{for some nonuero constant } \lambda_{\gamma}.
\end{equation*}
{\bf (c)} the coefficients of ${\bf H}_{\gamma}(u)$ are asymptotically given by
\begin{eqnarray*}
[u^n]{\bf H}_{\gamma}(u) & \sim &  c_\gamma \, n^{-3/2}\,\rho_{\gamma}^{-n}
\end{eqnarray*}
for some $c_\gamma>0$.
\end{theorem}
%%%
%%%%%%%%%%%%%%%%%%%%%%%%%%%%%%%%%%%%%%%%%%%%%%%%%%%%%%%%%%%%%%%%%%%%%
%%%
\begin{proof}
Pringsheim¡¯s Theorem (\cite{Flajolet:07a} pp. 240) guarantees that for any $\gamma$, ${\bf H}_{\gamma}(u)$ has a
dominant real singularity $\rho_{\gamma}>0$.
To prove the singular expansion of the function and asymptotic of the coefficients, we verify $P_\gamma(u,X)$,
for $1\leq \gamma \leq 10$, satisfy the condition of Theorem~\ref{T:AsymG} and the results follow.
\end{proof}

%%%%
%%%%%%%%%%%%%%%%%%%%%%%%%%%%%%%%%%%%%%%%%%%%%%%%%%%%%%%%%%%%%%%%%%%%%%%%
%%%%
\newpage
\section{Combinatorics of $\gamma$-diagrams}

%%%%
%%%%%%%%%%%%%%%%%%%%%%%%%%%%%%%%%%%%%%%%%%%%%%%%%%%%%%%%%%%%%%%%%%%%%%%%
%%%%

%%%%
%%%%%%%%%%%%%%%%%%%%%%%%%%%%%%%%%%%%%%%%%%%%%%%%%%%%%%%%%%%%%%%%%%%%%%%%
%%%%
\begin{lemma}\label{L:GFbi-sh}
For any $\gamma\geq 1$, we have
\begin{eqnarray}\label{E:GFsh}
{\bf S}_\gamma(u,e) & = & \frac{1+u}{1+2u-ue}
                    {\bf H}_\gamma\left(\frac{u(1+u)}{(1+2u-ue)^2}\right).
\end{eqnarray}
\end{lemma}
%%%
%%%%%%%%%%%%%%%%%%%%%%%%%%%%%%%%%%%%%%%%%%%%%%%%%%%%%%%%%%%%%%%%%%%%%%%%%%%
%%%

The proof of Lemma~\ref{L:GFbi-sh} can be obtained by standard symbolic method.
%%%%%%%%%%%%%%%%%%%%%%%%%%%%%%%%%%%%%%%%%%%%%%%%%%%%%%%%%%%%%%%%%%%%%%%%
%%%
\begin{lemma} \label{L:GFbull}
Let $\lambda$ be a fixed $\gamma$-shape with $s\geq 1$ arcs and
$m\geq 0$ 1-arcs. Then the generating function of $\tau$-canonical
$\gamma$-diagrams containing no $1$-arc that have shape $\lambda$ is given by
$$
{\bf G}^{\lambda}_{\tau,\gamma}(z)=(1-z)^{-1}\left(\frac{z^{2\tau}}
{(1-z^2)(1-z)^2-(2z-z^2)z^{2\tau}}\right)^s \, z^m.$$
In particular, ${\bf G}^\lambda_{\tau,\gamma} (z)$ depends only upon the
number of arcs and $1$-arcs in $\lambda$.
\end{lemma}
%%%
%%%%%%%%%%%%%%%%%%%%%%%%%%%%%%%%%%%%%%%%%%%%%%%%%%%%%%%%%%%%%%%%%%%%%%%%
%%%

Our main result about enumerating $\tau$-canonical $\gamma$-structures
follows.

%%%
%%%%%%%%%%%%%%%%%%%%%%%%%%%%%%%%%%%%%%%%%%%%%%%%%%%%%%%%%%%%%%%%%%%%%%%%
%%%
\begin{theorem}\label{T:genus}
Suppose $\gamma,\tau\geq 1$  and let  $u_\tau(z)
=\frac{(z^2)^{\tau-1}}{z^{2\tau}-z^2+1}$. Then the generating function
${\bf G}_{\tau,\gamma}(z)$ is algebraic and given by
\begin{eqnarray}\label{E:oho2}
{\bf G}_{\tau,\gamma}(z) & = & \frac{1}{u_\tau(z) z^2-z+1}\
                            {\bf H}_\gamma\left(\frac{u_\tau(z)z^2}
                            {\left(u_\tau(z) z^{2}-z+1\right)^2}\right).
\end{eqnarray}
In particular for $1\le s,i\le 2$ we have
\begin{equation*}
[z^n]{\bf G}_{s,i}(z)  \sim   k_{s,i} \,n^{-\frac{3}{2}} (\rho_{s,i}^{-1})^n,\quad
\end{equation*}
for some constants $k_{s,i}>0$, for $\rho_{s,i}^{-1}$,
we have Table~\ref{T:growthrate}.
\end{theorem}
%%%%%%%%%%%%%%%%%%%%%%%%%%%%%%%%%%%%%%

%%%%%%%%%%%%%%%%%%%%%%%%%%%%%%%%%%%%%%%%%%%%%%%%%%%%%%%%%%%%%%%%%%%%%%%%
%%%
\begin{proof}
Since each $\gamma$-diagram has a unique $\gamma$-shape,
$\lambda$, having some number $m\geq 0$ of $1$-arcs, we have
\begin{equation}\label{E:Hgf}
{\bf G}_{\tau,\gamma}(z) =
\sum_{m\geq 0}\sum_{\lambda\,  \text{\rm $\gamma$-shape}
\atop  \text{\rm having $m$ $1$-arcs}}
\mathbf{G}^\lambda_{\tau,\gamma}(z).
\end{equation}
According to Lemma \ref {L:GFbull}, ${\bf G}^\lambda_{\tau,\gamma}(z)$ only
depends on the number of arcs and $1$-arcs of $\lambda$, and we can
therefore express
\begin{eqnarray*}
{\bf G}_{\tau,\gamma}(z)
 & = &
{1\over{z-1}}~{\bf S}_\gamma\biggl({{z^{2\tau}}\over{(1-z^2)(1-z)^2-
(2z-z^2)z^{2\tau}}},z\biggr)\\
 & = &\frac{1}{(1-z) +{u_\tau(z)}z^2}\,
{\bf H}_\gamma\left(\frac{z^2\, {u_\tau(z)}}
{\bigl((1-z) +{u_\tau(z)}z^2
\bigr)^2}\right),
\end{eqnarray*}
using Lemma \ref{L:GFbi-sh} in order to confirm eq.~(\ref{E:oho}),
where the second equality follows from direct computation.
Let
\begin{equation*}\label{E:formula}
\theta_\tau(z)=\frac{z^2\, {u_\tau(z)}}
{\bigl((1-z) +{u_\tau(z)}z^2
\bigr)^2}
\end{equation*}
denote the argument of ${\bf H}_\gamma$ in this expression. By definition
we have $\theta(z)\in \mathbb{C}(z)$.
Since $\theta_\sigma(0)=0$ the composition ${\bf H}_\gamma(\theta(z))$ is
welldefined as a powerseries. Obviously, $P_\gamma(z,\mathbf{H}_\gamma(z))
=0$ guarantees $P_\gamma(\theta_\tau(z), {\bf H}_\gamma(\theta_\tau(z))=0$.
We have the following Hasse diagram of fields
\begin{equation*}
\diagram
   & \mathbb{C}(z,\theta_\tau(z),{\bf H}_\gamma(\theta_\tau(z)))& \\
\mathbb{C}(z,\theta_\tau(z))\urline &&    \mathbb{C}(z,{\bf H}_\gamma(z)) \\
                       &  \ulline\mathbb{C}(z) \urline &
\enddiagram
\end{equation*}
from which we immediately conclude that ${\bf G}_{\tau,\gamma}(z)$
is algebraic. Pringsheim's Theorem \cite{Flajolet:07a}
guarantees that for any $\gamma,\tau\ge 1$, ${\bf G}_{\tau,\gamma}(z)$
has a dominant real singularity $\rho_{\tau,\gamma}>0$.

According to Theorem \ref{T:HUniAsy} we have
\begin{equation*}
\mathbf{H}_i(z)=\pi_1 +
\sum_{j\ge 1} a_{j,i}\left(\left(\mu_i-z\right)^{1/2}\right)^{j}
\quad
\text{\rm and}\quad
[z^n]\mathbf{H}_i(z)\sim k_i \; n^{-3/2}\;
\left(\mu_i^{-1}\right)^n.
\end{equation*}
For $\tau=1,2$, we verify directly that $\rho_{1,i}$ and $\rho_{2,i}$ are
the unique solutions of minimum modulus of $\theta_1(z)=\mu_i$ and
$\theta_2(z)=\mu_i$. These solutions are strictly smaller than
any other singularities of $\theta_1(z)$ and $\theta_2(z)$ and
furthermore satisfy $\theta_1'(\rho_{1,i})\neq 0$ as well as
$\theta_2'(\rho_{2,i})\neq 0$.
It follows that ${\bf G}_{1,i}(z)$ and ${\bf G}_{2,i}(z)$ are governed by
the supercritical paradigm \cite{Flajolet:07a}, which in turn implies
\begin{equation}\label{E:singprime}
[z^n] \mathbf{G}_{s,i}(z) \sim k_{s,i} \,  n^{-3/2}\left(\rho_{s,i}^{-1}\right)^n
\end{equation}
where $s=1,2$ and $k_{s,i}$ is some positive constant.
\end{proof}

Theorem~\ref{T:genus} has its analogue for $\tau$-canonical, $\gamma$-diagrams
containing $1$-arcs. The asymptotic formula in case of $\tau=1,\gamma=1$,
$$
[z^n]\widetilde{{\bf G}}_{1,1}(z)  \sim   j_1\,n^{-\frac{3}{2}} (\varrho_{1,1}^{-1})^{n}
$$
is due to \cite{NebelWeinberg} who used the explicit grammar developed in
\cite{gfold} in order to obtain an algebraic equation for $\widetilde{{\bf G}}_{1,1}(z)$.

%%%
%%%%%%%%%%%%%%%%%%%%%%%%%%%%%%%%%%%%%%%%%%%%%%%%%%%%%%%%%%%%%%%%%%%%%%%%
%%%
\begin{corollary}\label{C:genus}
Suppose $\gamma,\tau\geq 1$  and let  $u_\tau(z)
=\frac{(z^2)^{\tau-1}}{z^{2\tau}-z^2+1}$. Then the generating function
of $\tau$-canonical $\gamma$-diagrams containing $1$-arcs,
$\widetilde{{\bf G}}_{\tau,\gamma}(z)$, is algebraic
and
\begin{eqnarray}\label{E:oho}
\widetilde{{\bf G}}_{\tau,\gamma}(z) & = &
{\bf H}_\gamma\left(\frac{u_\tau(z) z^2}{(1-z)^2}\right).
\end{eqnarray}
In particular for $\gamma=1$ we have
\begin{equation*}
[z^n]\widetilde{{\bf G}}_{1,1}(z)  \sim   j_1\,n^{-\frac{3}{2}}
(\varrho_{1,1}^{-1})^{n},\quad
\text{\rm and}\quad
[z^n]\widetilde{{\bf G}}_{2,1}(z)  \sim   j_2\,n^{-\frac{3}{2}} (\varrho_{2,1}^{-1})^{n}
\end{equation*}
for some constants $j_1,j_2$, where $\varrho_{1,1}^{-1}=3.8782$ and
$\varrho_{2,1}^{-1}=2.3361$.
\end{corollary}

\begin{proof}
Let $\lambda$ be a fixed $\gamma$-shape with $s\geq 1$ arcs and
$m\geq 0$ 1-arcs. Then the generating function of $\tau$-canonical
$\gamma$-diagrams containing $1$-arcs that have shape $\lambda$
containing 1-arcs is given by
$$
\widetilde{{\bf G}}_{\tau,\gamma}(z)=(1-z)^{-1}\left(\frac{z^{2\tau}}
{(1-z^2)(1-z)^2-(2z-z^2)z^{2\tau}}\right)^s .$$
\end{proof}
%%%
%%%%%%%%%%%%%%%%%%%%%%%%%%%%%%%%%%%%%%%%%%%%%%%%%%%%%%%%%%%%%%%%%%%%%%%%
%%%

%%%
%%%%%%%%%%%%%%%%%%%%%%%%%%%%%%%%%%%%%%%%%%%%%%%%%%%%%%%%%%%%%%%%%%%%%%%%%
%%%
\newpage
\section{Discussion}\label{S:discuss}
%%%
%%%%%%%%%%%%%%%%%%%%%%%%%%%%%%%%%%%%%%%%%%%%%%%%%%%%%%%%%%%%%%%%%%%%%%%%%
%%%

The symbolic approach based on $\gamma$-matchings allows not only to
compute the generating function of canonical $\gamma$-structures.
On the basis of Theorem~\ref{T:genus} it is possible to obtain a
plethora of statistics of $\gamma$-structures by means of combinatorial
markers.

For instance, we can analogously compute the bivariate generating function
of $\tau$ canonical $\gamma$-structures over $n$ vertices, containing
exactly $m$ arcs, ${\bf A}_{\tau,\gamma}(z,t)$ as
\begin{equation}\label{E:qw}
{\bf A}_{\tau,\gamma}(z,t) = \frac{1}{u_\tau(z,t) z^2-z+1}{\bf H}_\gamma\left(
\frac{u_\tau(z,t)\; z^2}{(u_\tau(z,t) z^2-z+1)^2}\right)
\end{equation}
where $u_\tau(z,t)$ is given by
\begin{equation*}
u_\tau(z,t)=\frac{t\,(tz^2)^{\tau-1}}{(tz^2)^{\tau}-tz^2+1} \ .
\end{equation*}
This bivariate generating function is the key to obtain a central limit
theorem for the distribution of arc-numbers in $\gamma$-structures
\cite{Bender:73} on the basis of L\'{e}vy-Cram\'{e}r Theorem on limit
distributions \cite{Feller}.

Statistical properties of $\gamma$-structures play a key role for quantifying
algorithmic improvements via sparsifications \cite{Backofen, Mohl, Wexler}.
The key property here is the {\it polymer-zeta property}
\cite{polymer-zeta1, polymer-zeta2} which states that the probability
of an arc of length $\ell$ is bounded by $k\, \ell^c$, where $k$ is some
positive constant and $c>1$.
Polymer-zeta stems from the theory of self-avoiding walks \cite{SAW} and has
only been empirically established for the simplest class of RNA structures,
namely those of genus zero.
It turns out however, that the polymer-zeta property is genuinely a combinatorial
property of a structure class. Moreover our results allow to quantify the effect
of sparsifications of folding algorithms into $\gamma$-structures
\cite{fenix2bb,fenixnewsparsi}.

We finally remark that around 98\% of RNA pseudoknot structures catalogued in
databases are in fact canonical $1$-structures. RNA pseudoknot structures
like the
\textsc{HDV}-virus\footnote{{\small {\tt www.ekevanbatenburg.nl/PKBASE/PKB00075.HTML}}}
exhibiting irreducible shadows of genus two are relatively rare.

%%%%%%%%%%%%%%%%%%%%%%%%%%%%%%%%%%%%%%%%%%%%%%%%%%%%%%%%%%%%%%%%%%%%%%%%%%
%%%
{\bf Acknowledgments.}
%%%
%%%%%%%%%%%%%%%%%%%%%%%%%%%%%%%%%%%%%%%%%%%%%%%%%%%%%%%%%%%%%%%%%%%%%%%%%%
%%%
We want to thank Fenix W.D.~Huang for discussions and comments.
We furthermore acknowledge the financial support of the Future and Emerging
Technologies (FET) programme within the Seventh Framework Programme (FP7) for
Research of the European Commission, under the FET-Proactive grant agreement
TOPDRIM, number FP7-ICT-318121.

%%%%%%%%%%%%%%%%%%%%%%%%%%%%%%%%%%%%%%%%%%%%%%%%%%%%%%%%%%%%%%%
%%%%%%%%%%%%%%%%%%%%%%%%%%%%%%%%%%%%%%%%%%%%%%%%%%%%%%%%%%%%%%%%%%%

%%%%
%%%%%%%%%%%%%%%%%%%%%%%%%%%%%%%%%%%%%%%%%%%%%%%%%%%%%%%%%%%%%%%%%%%%%%%%
%%%%

% BibTeX users please use one of
\bibliographystyle{plainnat}      % basic style, author-year citations
%\bibliography{thomas}   % name your BibTeX data base

%%%%%%%%%%%%%%%%%%%%%%%%%%%%%%%%%%%%%
\newpage
\begin{table}
% table caption is above the table
\caption{\small The exponential growth rates of
$\rho_{s,i}^{-1}$, for $1\le s,i\le 2$.}
\label{T:growthrate}      % Give a unique label
% For LaTeX tables use
\begin{tabular}{lllll}
\hline\noalign{\smallskip}
$(s, i)$ & $(1, 1)$ & $(2, 1)$ & $(1, 2)$ & $(2, 2)$ \\
\noalign{\smallskip}\hline\noalign{\smallskip}
$\rho_{s,i}^{-1}$ & 3.6005 & 2.2759 & 3.8846 & 2.3553 \\
\noalign{\smallskip}\hline
\end{tabular}
\end{table}
%%%
%%%%%%%%%%%%%%%%%%%%%%%%%%%%%%%%%%%%%%%%%%%%%%%%%%%%%%%%%%%%%%%%%%%%%%%%

\end{document}